\documentclass[12pt,a4paper]{article}
\usepackage[utf8]{inputenc}
\usepackage{amsmath,amsthm}
\usepackage[dvipsnames]{xcolor}
\usepackage{amsfonts}
\usepackage{amssymb}
\usepackage{tikz}
\usepackage{tikz-cd}
\usepackage{caption}
\usepackage[nottoc]{tocbibind}
\usepackage[toc,page]{appendix}
\usepackage{listings}
\usepackage{pdfpages}
\usepackage{mathtools}
\usepackage[all,cmtip]{xy}
\usepackage{lipsum}
\usepackage{breqn}
\usepackage{subcaption}
\usepackage{hyperref}
\usepackage{tkz-euclide}
\usepackage{pgfplots}
\usepackage[sort,nocompress]{cite}
\usepackage{comment}

\usetikzlibrary{angles,quotes}
\usetikzlibrary{decorations.markings}
\usetikzlibrary{hobby}

\usepackage{authblk}

\newtheorem{theorem}{Theorem}
\newtheorem{definition}{Definition}
\newtheorem{lemma}{Lemma}

\theoremstyle{remark}
\newtheorem*{remark}{Remark}

\numberwithin{lemma}{section}
\numberwithin{claim}{section}
\numberwithin{definition}{section}
\numberwithin{proposition}{section}
\numberwithin{equation}{section}

\def\N{\mathbb{N}}

\def\R{\mathbb{R}}

\def\<={\leq}
\def\>={\geq}
\def\inv{^{-1}}

\pgfplotsset{compat=1.18}

\newcommand{\set}[1]{\left\lbrace#1\right\rbrace}

\newcommand{\norm}[1]{\lVert #1 \rVert}
\newcommand{\normC}[2]{\norm{#1}_{C^{#2}}}
\newcommand{\normL}[2]
{\norm{#1}_{L^{#2}}}

\title{Local rigidity for symplectic billiards}

\author[1]{Daniel Tsodikovich\thanks{   Institute of Science and Technology Austria, 
Am Campus 1, 3400 Maria Gugging, Austria. \\
    \textit{E-mail}: \href{mailto:Daniel.Tsodikovich@ist.ac.at}{\texttt{Daniel.Tsodikovich@ist.ac.at}}}}

\date{}
\begin{document}
\maketitle
\begin{abstract}
We show a local rigidity result for the integrability of symplectic billiards. 
We prove that any domain which is close to an ellipse, and for which the symplectic billiard map is rationally integrable must be an ellipse as well.
This is in spirit of the result of \cite{avila2016integrable}  for Birkhoff billiards.
\end{abstract}
\section{Introduction and the main result}\label{sec:intro}

Symplectic billiards were introduced by Albers and Tabachnikov in \cite{albers2017introducing} as an analogous model for the billiard dynamics, where the area is used for the variational principle, instead of the length (see Subsection \ref{subsec:sympBilliards}  for a more precise definition).
While the two systems might behave somewhat differently, the similarity between the definitions of the Birkhoff billiards system and the symplectic one has recently attracted some attention to try and address the classical questions for the billiard system about integrability and rigidity in this new setting.
For example, in \cite{baracco2024totally},  Baracco and Bernardi showed a total integrability rigidity result for symplectic billiards. 
In \cite{baracco2024bialy}, an analogous result of the integrability reuslt of Bialy and Mironov \cite{10.4007/annals.2022.196.1.2}  was shown for symplectic billiards. 
A symplectic billiard version of the result of \cite{de2017dynamical}  on spectral rigidity was done independently in \cite{baracco2024area}  and in \cite{fierobe2024deformational}.

In this work, we address the question of local rigidity of integrability, and show that near ellipses, the only rationally integrable domains (see Definition \ref{def:RationalIntegrability} below) are ellipses.
This is in spirit of the results of \cite{avila2016integrable}.
We prove the following result (for the definition of a domain being close to an ellipse see \eqref{eq:RadialDeformationDomain}
 below):

\begin{theorem}\label{thm:SympDynRigidity}
   Let $\mathcal{E}$  be an ellipse in $\R^2$.
   Then, for any $K>0$, there exists $\varepsilon>0$ such that for any domain  $\Omega\subseteq \R^2$ that is   $C^{127}$ $K$-close to $\mathcal{E}$ and $C^1$ $\varepsilon$-close to $\mathcal{E}$, if the symplectic billiard dynamics in $\Omega$  is rationally integrable, then $\Omega$ is an ellipse.
\end{theorem}
\begin{remark}
    It is not known if symplectic billiards has an analogue of ellipses for Birkhoff billiards.
    That is, domains in which symplectic billiards are integrable, but do not admit a total foliation of the phase space by invariant curves.
    The theorem implies that such domains, if they exist, either cannot be used to approximate ellipses, or do not have an invariant curve of some rational rotation number (similar to how billiard in an ellipses do not have an invariant curve of rotation number $\frac{1}{2}$).
\end{remark}
The affine nature of this system means that all ellipses behave the same way.
For this reason, it seems that looking for improvements of the result in the spirit of \cite{kaloshin2018local,huang2018nearly, koval2021local} might be challenging. 
These results rely on metric properties of ellipses which are not circles.
These properties cause the Birkhoff billiard dynamics in an ellipse to be different than the dynamics in a circle, and this difference is exploited in their proofs.
From the point of view of symplectic billiards, all ellipses are the same, and so we cannot adapt the same method that is used  for Birkhoff billiards.

\textbf{Structure of the paper:} in Section \ref{sec:bg}  we give some background about symplectic billiards, and set some notations.
In Section \ref{sec:ProofDynamic}  we give the proof of the theorem.
Section \ref{sec:ProofDynamic} starts with a general overview of the idea of the proof.

\subsection*{Acknowledgements}
The author would like to thank Corentin Fierobe, Vadim Kaloshin, Illya Koval and Yunzhe Li for useful discussions. 
The author would also like to thank the referee for useful remarks.
\section{Background}\label{sec:bg}
In this section, we recall briefly the notions discussed in this paper. 
We recall the definition of symplectic billiards.
We explain which parametrization we are going to use for the curves in our discussion.
We recall the definition of rational integrability.
Finally, we set the notation for the norms we use.
\subsection{Symplectic Billiards}\label{subsec:sympBilliards}
Symplectic billiards were introduced by Albers and Tabachnikov in \cite{albers2017introducing}. 
In the plane (which is the only case which is of interest for us in this work), it is defined as follows.
Given a smooth convex curve $\gamma(t)$, three points $\gamma(t_1)$, $\gamma(t_2)$, and $\gamma(t_3)$  are consecutive points of a symplectic billiard orbit if and only if the tangent at $\gamma(t_2)$  is positively parallel to the vector $\gamma(t_3)-\gamma(t_1)$, see Figure \ref{fig:SymplecticBilliardLaw}.
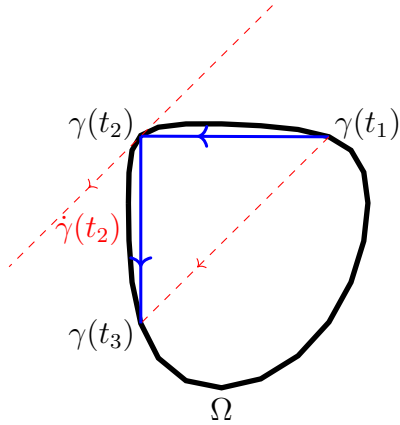
\begin{figure}
    \centering
    \begin{tikzpicture}[scale = 1.75]
        	\begin{scope}[decoration={
markings,
mark = at position 0.7 with {\arrow{>}}}
]

\draw[black, domain = -180:180, variable = \t, line width = 0.7mm] plot({cos(\t)*(1+0.2*cos(\t)-0.1*cos(2*\t))},{sin(\t)*(1-0.4*sin( \t)+0.01*cos(3*\t))});
\draw[line width = 0.4mm, postaction = {decorate},blue](0.807,0.502) -- (-0.607, 0.507);
\draw[red,dashed,postaction = {decorate}](-0.607+1.4*0.707,0.507+1.4*0.707)--(-0.607-1.4*0.707,0.507-1.4*0.707);
\draw[red,dashed, postaction={decorate}](0.807,0.502)--(0.807-2*0.707,0.502-2*0.707);
\draw[line width = 0.4mm, postaction = {decorate},blue](-0.607, 0.507)--(0.807-2*0.707,0.502-2*0.707);
\node[below] at (0,-1.4) {$\Omega$};
\node at (1.1,0.6) {$\gamma(t_1)$};
\node at (-0.9,0.6) {$\gamma(t_2)$};
\node[red] at (-1,-0.2) {$\dot{\gamma}(t_2)$};
\node at (-0.9,-1) {$\gamma(t_3)$};
\end{scope}
       \end{tikzpicture}

    \caption{The symplectic billiard collision law.}
    \label{fig:SymplecticBilliardLaw}
\end{figure}
Similarly to Birkhoff billiards, this system is also an example of a twist map of a cylinder (see, e.g., \cite{Bangert1988MatherSF,SEDP_1987-1988____A14_0, gole2001symplectic}), where the generating function is the area form on the plane $H(t,t')=\omega(\gamma(t),\gamma(t'))$, with $\omega$  denoting the area form (with respect to an arbitrary origin).
Namely, if we fix an arbitrary origin on the plane, and denote by $\omega$  the standard area form on $\R^2$, then one can check that the geometric definition above coincides with the following rule: $t_2$  is determined as the unique critical point of the function 
\[f(t)=H(t_1,t)+H(t,t_3),\]
for which the ordering of $\gamma(t_1)$, $\gamma(t_2)$, $\gamma(t_3)$  is consistent with the curve's orientation.
Indeed,
\begin{equation}\label{eq:VariationalDescription}
    f'(t)=0\Longleftrightarrow\omega(\gamma(t_3)-\gamma(t_1),\dot{\gamma}(t))=0,
\end{equation}
so the vectors are parallel.
Observe that the mixed partial derivative of $H$ is $\omega(\dot{\gamma}(t),\dot{\gamma}(t'))$. 
The symplectic billiard map is defined on the part of the phase space where the vectors $\gamma(t)$, $\gamma(t')$  are positively oriented (see \cite[Lemma 2.1]{albers2017introducing}), and hence, their tangents are also positively oriented, so the mixed partial derivative is positive and the symplectic billiards map is a negative twist map.
Hence, by the theory of twist maps (see e.g., \cite[Theorem 35.2]{gole2001symplectic}), all orbits on invariant curves are locally maximizing. 
\subsection{Affine arc-length parametrization}\label{subsec:Lazutkin}
It was shown in \cite[Theorem 3]{albers2017introducing}  that the affine arc-length parametrization (see, e.g., \cite{sapiro1994affine}) is suitable for the study of symplectic billiards. 
The billiard map in these coordinates has a good asymptotic form, which resembles the asymptotic form of Birkhoff billiards in Lazutkin coordinates (see, e.g., \cite[Section 14]{lazutkin2012kam}).
For ellipses, the affine arc-length parametirzation coincides with the usual trigonometric parametrization: if an ellipse with axes parallel to $x$  and $y$  axes has semiaxes lengths $a$, $b$, then its affine arc-length parametrization is
\[\gamma:[0,2\pi (ab)^{\frac{1}{3}}]\to\R^2\] \[\gamma(t)=\Big(a\cos\frac{t}{(ab)^{\frac{1}{3}}},b\sin\frac{t}{(ab)^{\frac{1}{3}}}\Big).\]
The domains we consider are perturbations of ellipses. 
Therefore we parametrize them using the affine arc-length parametrization of the ellipses they approximate (see \eqref{eq:RadialDeformationDomain}  for the precise parametrization). 
\subsection{Rational integrability}\label{subsec:RationalIntegrability}
The notion of \textit{integrability} that we use in this work is \textit{rational integrability}.
More precisely, we use the following definition.
\begin{definition}\label{def:RationalIntegrability}
    Let $T$  be the symplectic billiard map inside some planar convex domain, viewed as a self map of the phase cylinder.
    We say that $T$ is \emph{rationally integrable}  if for all $3\leq q\in\N$, there exists an invariant curve for $T$  the rotation number of which is $\frac{1}{q}$, and this curve consists entirely of $q$ -periodic orbits. 
\end{definition}
If an open set in the cylinder is $C^1$ foliated by invariant curves, then each of those curves inherits an absolutely continuous invariant measure.
If, in addition, the rotation number of the curve is rational, then all points on it must be periodic.
Therefore, the assumption that we use in this work is weaker than the assumption about the existence of a foliation by invariant curves.

Observe that we consider invariant curves of rotation number $\frac{1}{q}$ only for $q\geq 3$.
By definition, for any $t$, if we denote by $t^*$  the parameter for which $\dot{\gamma}(t)$, $\dot{\gamma}(t^*)$  are parallel, we have a $2$-periodic orbit between $\gamma(t)$  and $\gamma(t^*)$. 
Hence, in the phase cylinder, we always have an invariant curve of two-periodic orbits. 
For this reason, in what follows, assumptions on the invariant curve of two-periodic orbits are superfluous.

\subsection{Choices of norms}\label{subsec:Norms}
In this subsection we fix notation and write explicitly what are the various norms that we use.
\begin{itemize}
    \item $L^2$  norm on the space of functions defined on the interval $[0,L]$: 
\[\normL{f}{2}^2=\frac{1}{L}\int_0^L |f(x)|^2 dx.\]
\item $C^k$  norm on $C^k$  functions on an interval $I$:
\[\normC{f}{k}=\sum_{j=0}^k\max_{x\in I} |f^{(j)}(x)|.\]
\item For an affine map $T:\R^n\to\R^n$, $Tx=Ax+b$, the norm is \[\norm{T}=\norm{A}+\norm{b},\]where $\norm{A}$  denotes the operator norm of the matrix.
\end{itemize}
\section{Proof of Theorem \ref{thm:SympDynRigidity}}\label{sec:ProofDynamic}

The proof mimics very closely the idea of the proof in \cite{avila2016integrable}.
We begin by giving the general plan of the proof:
\begin{enumerate}
    \item \label{itm:ActionDeviation} Estimate how much the action of a $q$- periodic orbit deviates from the action of a $q$- periodic orbit for an ellipse (Lemma \ref{lem:quadraticEstimate}).
    \item \label{itm:FourierFromAction} Use this estimate to derive a bound on the Fourier coefficients of the deformation function (Lemma \ref{lem:FourierDecayIntermediate}).
    \item \label{itm:Approximation} Given a rationally integrable perturbation of an ellipse, we use our estimates to find another ellipse.
This ellipse should be even closer  than the given one (Lemma \ref{lem:FindBetterApproxEllipse}).
    \item This allows us to finish the proof in Subsection \ref{subsec:FinishProofOfLocalBirkhoff}, by considering the ``closest" ellipse. 
\end{enumerate}

We first start with the following notation.
\begin{definition}\label{def:NormalizedArea}
    For an ellipse $\mathcal{E}$, we call the quantity $\alpha(\mathcal{E})=\sqrt[3]{\frac{1}{\pi}\mathrm{area}(\mathcal{E})}$  the \emph{normalized area} of $\mathcal{E}$. 
\end{definition}
Note that if an ellipse $\mathcal{E}$  has semiaxes $a$,$b$  (with respect to any Euclidean metric) then $\alpha(\mathcal{E})=(ab)^{\frac{1}{3}}$.
This quantity is not really consequential for us.
It is needed for technical reasons: during the proof we consider several ellipses, and thus we need to understand how our estimates vary with the ellipse. 
Since all of our ellipses will be close to each other, their normalized areas will be close to each other, and we will conclude that the upper bounds we use are uniform for this family of ellipses.

Given such an ellipse, we parametrize it by affine arc-length parametrization:
\[e_{a,b}(t)=(a\cos\frac{t}{A},b\sin\frac{t}{A}),\]
where $A=\alpha(\mathcal{E})$ , and $t\in[0,2\pi A]$. 
We consider a deformation of this ellipse into a domain along the affine normal vector, $A^{-2}e_{a,b}(t)$ (see \cite{sapiro1994affine}), but it is more convenient to consider a rescaling, $N(t)=e_{a,b}(t)$. 
 This way, given a function $n(t)$, we get a deformed domain $\Omega$  the boundary of which is given by
  \begin{equation}\label{eq:RadialDeformationDomain}
\partial \Omega = \set{\gamma(t)=e_{a,b}(t)+n(t)N(t)\mid t \in [0,2\pi A]}.    
\end{equation}
As long as the function $n$  is small enough, the curve $\gamma(t)$ is a convex curve.
We express this in a similar way to \cite{avila2016integrable}, and write
\[\partial\Omega = \mathcal{E}+n.\]
The notion of being ``close to an ellipse" is understood in terms of this function $n$.
Our integrability assumption means that for all $q>2$ there are invariant curves of $q$-periodic orbits for the symplectic billiard map.
The action of orbits on a given invariant curve is constant.
In our case, this action is twice the area enclosed by the orbit.
Hence, our first goal would be to understand the deviation of the area of this $q$-gon from the elliptic case.
\subsection{Second order error term for areas}\label{subsec:SecondOrderArea}
Our goal in this subsection is to prove the following lemma. 
\begin{lemma}\label{lem:quadraticEstimate}
    Suppose $\Omega$ is a domain of the form \eqref{eq:RadialDeformationDomain}, where $A=\alpha(\mathcal{E})$.
    Then there exist $\varepsilon, C, N>0$ such that for all functions $n$ with $C^1$ norm smaller than $\varepsilon$, and for each $2<q\in\N$, if $\normC{n}{1}<N \min\set{1,A^{-4}}q^{-8}$ and if the symplectic billiard map in $\Omega$ has a $q$  periodic orbit starting from $\gamma(t_0)$, for some $t_0\in[0,2\pi A]$, then
    \begin{equation}\label{eq:SecondOrderBound}
        |A_{q,\Omega}-A^3 q\sin\frac{2\pi}{q}-2A^3 \sin\frac{2\pi}{q}\sum_{j=0}^{q-1}n(t_0+\frac{2\pi Aj}{q})|\leq C\max\set{A^3,A^{9}}q^{31}(1+\normC{n}{1})^8\normC{n}{1}^2.
    \end{equation}
    Here $A_{q,\Omega}$ denotes the action of the $q$-periodic symplectic billiard orbit starting at $\gamma(t_0)$.
\end{lemma}
\begin{proof}
    From the definition of symplectic billiards, if a $q$-periodic symplectic billiard orbit  passes through the points $\gamma(t_j)$, for $j=0,...,q$ (and $t_0=t_q$), then its action is
    \[A_{q,\Omega}=\sum_{j=0}^{q-1}\omega(\gamma(t_j),\gamma(t_{j+1})).\]
    Using the fact that $\gamma(t)=(1+n(t))e_{a,b}(t)$, we get:
    \begin{gather*}
        A_{q,\Omega} = \sum_{j=0}^{q-1}(1+n(t_{j+1}))(1+n(t_j)) A^3\sin(\frac{t_{j+1}-t_j}{A}) = \\
        =A^3\sum_{j=0}^{q-1}\sin(\frac{t_{j+1}-t_j}{A})+(n(t_{j+1})+n(t_j))\sin(\frac{t_{j+1}-t_j}{A})+ n(t_{j+1})n(t_j)\sin(\frac{t_{j+1}-t_j}{A}).
    \end{gather*}
    When $n$ is the zero function, then $t_j=t_0+\frac{2\pi jA}{q}$, $j=0,...,q-1$, so for small $n$, the values of $t_j$ should be close to equidistributed.
    This is made precise by the following lemma that we prove later.
    \begin{lemma}\label{lem:estimatingJumps}
       Given an ellipse $\mathcal{E}$ with $A=\alpha(\mathcal{E})$, there exist  $C,C',N>0$ with the following property: for all functions $n$, if $\Omega$  is given by \eqref{eq:RadialDeformationDomain}, and $\gamma(t_0),...,\gamma(t_{q-1})$  is a $q$-periodic symplectic billiard orbit in $\Omega$, if $\normC{n}{1}\leq N\min\set{1,A^{-4}}q^{-8}$, then there exist $X_0,...,X_{q-1}$  for which for all $j=0,...,q-1$:
        \begin{equation}\label{eq:tEstimate}
            |t_j-(t_0+\frac{2\pi jA}{q})| \leq C \max\set{A,A^2}q^4 \normC{n}{1}(1+\normC{n}{1}),
        \end{equation}
        \begin{equation}\label{eq:sineEstimate}
            |\sin(\frac{t_{j+1}-t_j}{A})-\sin\frac{2\pi}{q}-\frac{1}{A}X_j\cos\frac{2\pi}{q}|\leq Cq^{30}\max\set{1,A^{6}}\normC{n}{1}^2(1+\normC{n}{1})^8.
        \end{equation}
        The numbers $X_j$ satisfies moreover, $|X_j|\leq C' \max\set{A,A^{4}}q^{15}\normC{n}{1}(1+\normC{n}{1})^4$, and $\sum_{j=0}^{q-1}X_j=0$. 
    \end{lemma}
    Now take an $N$ as in Lemma \ref{lem:estimatingJumps}, and suppose that $\normC{n}{1}\leq N\min\set{1,A^{-4}}q^{-8}$.
    We estimate the difference
    \begin{multline}\label{eq:ActionDeviation0}
        |A_{q,\Omega}-A^3 q\sin\frac{2\pi}{q} -2A^3 \sin\frac{2\pi}{q}\sum_{j=0}^{q-1}n(t_0+\frac{2\pi Aj}{q})| \leq \\
        \leq A^3 |\sum_{j=0}^{q-1}[\sin(\frac{t_{j+1}-t_j}{A})-\sin\frac{2\pi}{q}]| +\\
        +A^3|\sum_{j=0}^{q-1}[(n(t_{j+1})+n(t_j))\sin(\frac{t_{j+1}-t_j}{A})-2\sin\frac{2\pi}{q}n(t_0+\frac{2\pi j}{q})]|+\\
        +A^3 \sum_{j=0}^{q-1}|n(t_j)n(t_{j+1})\sin(\frac{t_{j+1}-t_j}{A})|,
    \end{multline}
    and deal with each row separately.
    For the first row of \eqref{eq:ActionDeviation0}, using \eqref{eq:sineEstimate}, and the fact that $\sum_{j=0}^{q-1}X_j=0$ we can write
    \begin{gather*}A^3|\sum_{j=0}^{q-1}\sin(\frac{t_{j+1}-t_j}{A})-\sin\frac{2\pi}{q}|=A^3|\sum_{j=0}^{q-1}[\sin(\frac{t_{j+1}-t_j}{A})-\sin\frac{2\pi}{q}-\frac{1}{A}X_j\cos\frac{2\pi}{q}]|\leq \\
    \leq C\max\set{A^3,A^{9}} q^{31}\normC{n}{1}^2(1+\normC{n}{1})^8.\end{gather*}
    For the last line of \eqref{eq:ActionDeviation0}, we can bound it by
    \[A^3 \sum_{j=0}^{q-1}|n(t_j)n(t_{j+1})\sin(\frac{t_{j+1}-t_j}{A})|\leq A^3 q\normC{n}{1}^2.\]
    For the middle row of \eqref{eq:ActionDeviation0}, first we write 
    \begin{multline}\label{eq:ExpansionLinearTerm}A^3\sum_{j=0}^{q-1}(n(t_{j+1})+n(t_j))\sin(\frac{t_{j+1}-t_j}{A})=A^3\sum_{j=0}^{q-1}(n(t_{j+1})+n(t_j))[\sin(\frac{t_{j+1}-t_j}{A})-\\
    -\sin\frac{2\pi}{q}-\frac{1}{A}\cos\frac{2\pi}{q}X_j+\sin\frac{2\pi}{q}+\frac{1}{A}\cos\frac{2\pi}{q}X_j]=A^3\sin\frac{2\pi}{q}\sum_{j=0} ^{q-1}(n(t_{j+1})+n(t_j)) + \\
    +A^3\sum_{j=0}^{q-1}(n(t_{j+1})+n(t_j))(\sin(\frac{t_{j+1}-t_j}{A})-\sin\frac{2\pi}{q}-\frac{1}{A}\cos\frac{2\pi}{q}X_j)+\\
    +A^2\cos\frac{2\pi}{q}\sum_{j=0}^{q-1}(n(t_j)+n(t_{j+1}))X_j=2A^3\sin\frac{2\pi}{q}\sum_{j=0}^{q-1}n(t_j)+\\
    +A^3\sum_{j=0}^{q-1}(n(t_{j+1})+n(t_j))(\sin(\frac{t_{j+1}-t_j}{A})-\sin\frac{2\pi}{q}-\frac{1}{A}\cos\frac{2\pi}{q}X_j)+\\
    +A^2\cos\frac{2\pi}{q}\sum_{j=0}^{q-1}(n(t_{j+1})+n(t_j))X_j.
    \end{multline}
    Now, when we consider the difference in the middle row of \eqref{eq:ActionDeviation0}, then we can subtract from the first summand in \eqref{eq:ExpansionLinearTerm} the term with $n(t_0+\frac{2\pi j}{q})$, and estimate the other summands separately.
    This way we get, using \eqref{eq:sineEstimate}, and the bound on $X_j$,
\begin{multline}\label{eq:ActionDeviation}
          |A^3\sum_{j=0}^{q-1}[(n(t_{j+1})+n(t_j))\sin(\frac{t_{j+1}-t_j}{A})-2\sin\frac{2\pi}{q}n(t_0+\frac{2\pi j}{q})]|\leq \\
        \leq 2A^3\sin\frac{2\pi}{q}\sum_{j=0}^{q-1}|n(t_j)-n(t_0+\frac{2\pi Aj}{q})| + \\
        +CA^3q\normC{n}{1}\cdot q^{30}\max\set{1,A^{6}}\normC{n}{1}^2(1+\normC{n}{1})^8+\\
        +CA^2 q\normC{n}{1}\cdot q^{15}\max\set{A,A^{4}}\normC{n}{1}(1+\normC{n}{1})^4.
  \end{multline}
  
Using \eqref{eq:tEstimate} we can bound $|n(t_j)-n(t_0+\frac{2\pi Aj}{q})|$ with \[C\max\set{A,A^2}q^4\normC{n}{1}^2(1+\normC{n}{1}).\] so the first sum is bounded (using  $\sin\frac{2\pi}{q}\leq\frac{2\pi}{q}$)  by a quantity of the same order.
    In total, the middle row of \eqref{eq:ActionDeviation0} is bounded by
    \begin{gather*}C\max\set{A^4,A^5}q^4\normC{n}{1}^2(1+\normC{n}{1})+C\max\set{A^3,A^{9}}q^{31}\normC{n}{1}^3(1+\normC{n}{1})^8+ \\
     +C\max\set{A^3,A^6}q^{16}\normC{n}{1}^2(1+\normC{n}{1})^4 \leq C\max\set{A^3,A^{9}}q^{31}\normC{n}{1}^2(1+\normC{n}{1})^8,
    \end{gather*}
    for $\normC{n}{1}$ small enough (for example for $\normC{n}{1}<1$).
    So we see that the first and second row of \eqref{eq:ActionDeviation0} are the largest, and they give us the error term required in \eqref{eq:SecondOrderBound}.
\end{proof}
Now we  prove Lemma \ref{lem:estimatingJumps}. 
The idea is that for $n=0$ we should have $t_j=t_0+\frac{2\pi A j}{q}$.
So for small $n$ we can estimate the deviation from this configuration using the implicit function theorem.
\begin{proof}[Proof of Lemma \ref{lem:estimatingJumps}]
    First we write the condition that the sequence $t_0,...,t_{q-1}$ corresponds to a symplectic billiard orbit.
    From the definition of symplectic billiards \eqref{eq:VariationalDescription}, we must have for all $j=1,...,q-1$,
    \[\omega(\gamma(t_{j+1})-\gamma(t_{j-1}),\dot{\gamma}(t_j))=0.\]
Note that we also have an equation for $j=0$:
\[\omega(\gamma(t_1)-\gamma(t_{q-1}),\dot\gamma(t_0))=0,\]
However, in what follows we do not use it.
    Using \eqref{eq:RadialDeformationDomain}, we can expand it (we also add an extra minus for convenience),
    \begin{multline}\label{eq:SymplecticOrbit}
        0=-\omega(\gamma(t_{j+1})-\gamma(t_{j-1}),\dot{\gamma}(t_j))=-\omega((1+ n(t_{j+1}))e_{a,b}(t_{j+1}) - \\
        -(1+n(t_{j-1}))e_{a,b}(t_{j-1}),n'(t_j)e_{a,b}(t_j)+(1+n(t_j))\dot{e}_{a,b}(t_j))=\\
        =A^3(1+n(t_{j+1}))n'(t_j)\sin(\frac{t_{j+1}-t_j}{A})+A^3(1+n(t_{j-1}))n'(t_j)\sin(\frac{t_j-t_{j-1}}{A})-\\
        -A^2(1+n(t_{j+1}))(1+n(t_j))\cos(\frac{t_{j+1}-t_j}{A})+\\
        +A^2(1+n(t_j))(1+n(t_{j-1}))\cos(\frac{t_j-t_{j-1}}{A}).
    \end{multline}
    Let us define now a function $F:\R^{q-1}\times \R^{q-1}\times \R^{q-1}\to \R^{q-1}$ by:
    \[F = (F_1,...,F_{q-1}),\]
    where,
    \begin{gather*}F_j(a_1,...,a_{q-1},b_1,...,b_{q-1},t_1,...,t_{q-1}) = A^3(1+a_{j+1})b_j\sin(\frac{t_{j+1}-t_j}{A}) + \\
    +A^3(1+a_{j-1})b_j\sin(\frac{t_j-t_{j-1}}{A}) -A^2 (1+a_{j+1})(1+a_j)\cos(\frac{t_{j+1}-t_j}{A})+\\
    +A^2(1+a_j)(1+a_{j-1})\cos(\frac{t_j-t_{j-1}}{A}).
    \end{gather*}
    Here, the terms $a_0=a_q$ should be understood as the constant $n(t_0)$.
    Then \eqref{eq:SymplecticOrbit}  holds if and only if
    \[F(n(t_1),...,n(t_{q-1}),n'(t_1),...,n'(t_{q-1}),t_1,...,t_{q-1})=0.\]
    In particular, it holds that $F(0,0,\mathbf{t_0})=0$, where $\mathbf{t_0}=(t_0+\frac{2\pi A}{q},...,t_0+\frac{2\pi A(q-1)}{q})\in\R^{q-1}$.
    We show that the equation $F=0$ defines $t=(t_1,...,t_{q-1})$ as a function of $(a,b)=(a_1,...,a_{q-1},b_1,...,b_{q-1})$ in a neighborhood of $(0,0,\mathbf{t_0})$.
    For that we need to check that $\frac{\partial F}{\partial t}(0,0,\mathbf{t_0})$ is non-degenerate.
    One computes that
    \[\frac{\partial F_j}{\partial t_k} = \begin{cases}
        0, & |j-k|>1 \\
        -A^2b_j(1+a_{j-1})\cos(\frac{t_j-t_{j-1}}{A})+A(1+a_{j-1})(1+a_j)\sin(\frac{t_j-t_{j-1}}{A}), & k=j-1 \\
        A^2b_j(1+a_{j+1})\cos(\frac{t_{j+1}-t_j}{A})+A(1+a_j)(1+a_{j+1})\sin(\frac{t_{j+1}-t_j}{A}), & k = j+1 \\
        -b_jA^2(1+a_{j+1})\cos(\frac{t_{j+1}-t_j}{A}) -A (1+a_j)(1+a_{j+1})\sin(\frac{t_{j+1}-t_j}{A}) + \\
        +A^2b_j(1+a_{j-1})\cos(\frac{t_j-t_{j-1}}{A})-A(1+a_{j-1})(1+a_j)\sin(\frac{t_j-t_{j-1}}{A}), & k=j.
    \end{cases}\]
    When $a=0,b=0,t=\mathbf{t_0}$, we get:
    \[\frac{\partial F}{\partial t}=A\sin\frac{2\pi}{q}\begin{pmatrix}
        -2 & 1 & 0 & \dots & 0 \\
        1 & -2 & 1 & \\
        0 & 1 & -2 & & \vdots\\
        \vdots & & & \ddots & 1\\
        0 &... & & 1 & -2
    \end{pmatrix}\]
    The eigenvalues of this matrix are well known (see, e.g., \cite{KULKARNI199963}), and the one that has the smallest absolute value is $A\sin\frac{2\pi}{q}(-2+2\cos\frac{2\pi}{q})$.
    Therefore in a neighborhood of $(0,0,\mathbf{t_0})$ the equation $F=0$  indeed defines $t$ as a function of $a$, $b$.
    However, we need an explicit neighborhood where it exists, since we wish to use this argument for many values of $q$ (see \eqref{eq:AsymptoticCatMouse}).
    Therefore we find explicit $r_x,r_y$  such that if  $|a_j|,|b_j|<r_x$, and $|t_j-(t_0+\frac{2\pi j}{q})|<r_y$ then $\frac{\partial F}{\partial t}(a,b,t)$  is non-degenerate, and the operator norm of the inverse is bounded by 
    \begin{equation}\label{eq:NormOfInverse}
        \norm{\Big(\frac{\partial F}{\partial t}(a,b,t)\Big)\inv} \leq  C\Big(A\sin\frac{2\pi}{q}(1-\cos\frac{2\pi}{q})\Big)\inv\leq C A\inv q^3.
    \end{equation}
    
    We used here the inequalities $\sin x\geq \frac{2}{\pi}x$ and $1-\cos x\geq \frac{1}{\pi}x^2$, and $C$ is some universal constant.
    The derivation of those $r_x, r_y$ is technical, and hence it is done at the end of the proof (see \eqref{eq:exprRx}, \eqref{eq:exprry}  below  for the definitions of $r_x$  and $r_y$).
    Thus, we have an explicit neighborhood where $F=0$  implicitly defines $t_j=G_j(a,b)$, and \eqref{eq:NormOfInverse}  holds in this neighborhood.
    It is also convenient to define $G_0,G_q$ as the constant functions $t_0$ and $t_0+2\pi A$ respectively, and write $G=(G_0,...,G_{q-1})$.
    Hence, if $t_0,...,t_{q-1}$ form a symplectic billiard orbit for $\gamma$, then
    \[(t_0,...,t_{q-1})=G\Big(n(t_1),...,n(t_{q-1}),n'(t_1),...,n'(t_{q-1})\Big),\]
    provided that the norm $\normC{n}{1}$ is small enough.
    To obtain \eqref{eq:tEstimate}, we use a first order approximation,
    \begin{equation}\label{eq:tFirstOrderApprox}
        |G(a,b)-\mathbf{t_0}|\leq \norm{\frac{\partial G}{\partial a}}\norm{a}+\norm{\frac{\partial G}{\partial b}}\norm{b}.
    \end{equation}
Here, the first order derivatives are evaluated at some unknown points in the domain $\norm{a}_\infty,\norm{b}_\infty  \leq r_x$.
    We use the implicit function theorem to compute the derivatives.
    \[\frac{\partial G}{\partial a}=-\Big(\frac{\partial F}{\partial t}\Big)\inv\frac{\partial F}{\partial a},\]
    and similarly for $b$.
    Thus we need to estimate the operator norm of $\frac{\partial F}{\partial a}$ and $\frac{\partial F}{\partial b}$.
    It holds that
    \[\frac{\partial F_j}{\partial a_k} = \begin{cases}
        0, & |j-k| > 1 \\
       A^3 b_j\sin(\frac{t_j-t_{j-1}}{A})+A^2(1+a_j)\cos(\frac{t_j-t_{j-1}}{A}), & k=j-1 \\
        A^3 b_j\sin(\frac{t_{j+1}-t_j}{A})-A^2(1+a_j)\cos(\frac{t_{j+1}-t_j}{A}), & k=j+1 \\
        -A^2(1+a_{j+1})\cos(\frac{t_{j+1}-t_j}{A})+A^2(1+a_{j-1})\cos(\frac{t_j-t_{j-1}}{A}), & k=j.
    \end{cases}\]
    This means that $|\frac{\partial F_j}{\partial a_k}|\leq 2\max\set{A^2,A^3}(1+\normC{n}{1}),$ so its operator norm is bounded by $2\max\set{A^2,A^3}\sqrt{q}(1+\normC{n}{1})$.
    Here we used the fact that this matrix is sparse: in every row, at most three entries are non-zero.
    Next, $\frac{\partial F}{\partial b}$ is diagonal and its diagonal entries are 
    \[A^3(1+a_{j+1})\sin(\frac{t_{j+1}-t_j}{A})+A^3(1+a_{j-1})\sin(\frac{t_j-t_{j-1}}{A}),\]
    so its operator norm is bounded by $2A^3(1+\normC{n}{1})$.
    Using \eqref{eq:NormOfInverse}, we get the estimates
    \begin{equation}\label{eq:BoundsFirstOrderDerG}
    \norm{\frac{\partial G}{\partial a}} \leq Cq^{3.5}\max\set{A,A^2}(1+\normC{n}{1}),\quad \norm{\frac{\partial G}{\partial b}}\leq C q^3A^2(1+\normC{n}{1}).\end{equation}
    Using \eqref{eq:tFirstOrderApprox} (with $a_j=n(t_j), b_j = n'(t_j)$), we get
    \[|t_j-(t_0+\frac{2\pi j}{q})|\leq Cq^{3.5}\max\set{A,A^2}(1+\normC{n}{1})\cdot \sqrt{q}\normC{n}{1}.\]
    This gives us \eqref{eq:tEstimate}.

    To obtain \eqref{eq:sineEstimate} we repeat the same idea, however now we use a quadratic approximation:
\begin{gather*}G_j(a,b)-(t_0+\frac{2A\pi j}{q}) = \frac{\partial G_j}{\partial a}a+\frac{\partial G_j}{\partial b}b + \\
+\frac{1}{2}\Big(\frac{\partial^2G_j}{\partial a^2}(a,a)+\frac{\partial^2 G_j}{\partial a\partial b}(a,b)+\frac{\partial^2 G_j}{\partial b\partial a}(b,a)+\frac{\partial^2 G_j}{\partial b^2}(b,b)\Big).
\end{gather*}
    Here the first order derivatives are evaluated at $(0,0)$, while the second order derivatives are evaluated at some unknown point for which $|a_j|,|b_j|\leq r_x$.
    Denote the second order term by $Z_j$. 
    Note that the above also holds for $j=0,q$, in which case $Z_0=Z_q=0$.
    Then
    \begin{equation}\label{eq:SecondOrderDeltaT}
        G_{j+1}(a,b)-G_j(a,b) - \frac{2\pi A}{q} = \frac{\partial (G_{j+1}-G_j)}{\partial a}a + \frac{\partial (G_{j+1}-G_j)}{\partial b}b + Z_{j+1}-Z_j.
    \end{equation}
    Denote by $X_j$ the right hand side in \eqref{eq:SecondOrderDeltaT}.
    Then it holds that $\sum_{j=0}^{q-1} X_j = 0$, since $G_0$ and $G_q$ are constant (and all first order derivatives are evaluated at the same point), and $\sum(Z_{j+1}-Z_j)=Z_q-Z_0=0$.
    Now we can use a first order estimate for the sine:
    \[\sin(\frac{t_{j+1}-t_j}{A})-\sin\frac{2\pi}{q}=\frac{1}{A}X_j\cos\frac{2\pi}{q}+\frac{1}{2A^2}X_j^2\sin(\xi),\]
    where $\xi$ is some unknown point.
    To get the desired estimate, it is then enough to show that 
    \begin{equation}\label{eq:DesiredErrorTermXZ}
        \frac{1}{2A^2}X_j^2\sin(\xi)
    \end{equation}
    can be bounded by the right hand side of \eqref{eq:sineEstimate}.
    The second derivatives of an implicit function can be found using the following formula:
    \begin{multline}\label{eq:SecondOrderImplicitDouble}
        \frac{\partial^2 G_m}{\partial a_k \partial a_s} = -\sum_{j=1}^{q-1}\Big((\frac{\partial F}{\partial t})\inv\Big)_{m,j}\Big{(}\frac{\partial^2 F_j}{\partial a_k\partial a_s}+\sum_{p=1}^{q-1}\Big[\frac{\partial^2 F_j}{\partial t_p\partial a_s}\frac{\partial G_p}{\partial a_k}+\frac{\partial^2 F_j}{\partial t_p \partial a_k}\frac{\partial G_p}{\partial a_s}\Big]+\\
        +\sum_{p,r=1}^{q-1}\frac{\partial^2 F_j}{\partial t_p\partial t_r }\frac{\partial G_p}{\partial a_k}\frac{\partial G_r}{\partial a_s}\Big{)}.
    \end{multline}
    Similarly for derivatives with respect to $b_k,b_s$ and similarly,
    \begin{multline}\label{eq:SecondOrderImplicitMixed}
        \frac{\partial^2 G_m}{\partial a_k \partial b_s} = -\sum_{j=1}^{q-1}\Big((\frac{\partial F}{\partial t})\inv\Big)_{m,j}\Big{(}\frac{\partial^2 F_j}{\partial a_k\partial b_s}+\sum_{p=1}^{q-1}\Big[\frac{\partial^2 F_j}{\partial t_p\partial b_s}\frac{\partial G_p}{\partial a_k}+\frac{\partial^2 F_j}{\partial t_p\partial a_s}\frac{\partial G_p}{\partial b_s}\Big]+\\
        +\sum_{p,r=1}^{q-1}\frac{\partial^2 F_j}{\partial t_p\partial t_s}\frac{\partial G_p}{\partial a_k}\frac{\partial G_r}{\partial b_s}\Big{)}.
    \end{multline}
    Using \eqref{eq:BoundsFirstOrderDerG}, the sum of all first order derivatives appearing in $X_j$ can be bounded by
    \[C\max\set{A,A^2} q^4\normC{n}{1}(1+\normC{n}{1}).\]
    If $M$ is an upper bound on all of the second order partial derivatives of $G_j$, then we have
    \[ |X_j|\leq C\max\set{A,A^2} q^4\normC{n}{1}(1+\normC{n}{1})+ 4M\normC{n}{1}^2q^2.\]
    The second order partial derivatives of $F_j$ that do not involve $t$ are bounded by $2\max\set{A^2,A^{3}}$,  those that involve $t$ with $a$ or $b$ are bounded by $2\max\set{A,A^2}(1+\normC{n}{1})$, and those that only involve $t$  are bounded by $2\max\set{1,A}(1+\normC{n}{1})^2$. 
    Thus,
    \begin{gather*}
    M\leq q \norm{\Big(\frac{\partial F}{\partial t}\Big)\inv}\biggl(2\max\set{A^2,A^{3}}+4\max\set{A,A^2} q(1+\normC{n}{1})\norm{\frac{\partial G_j}{\partial a}}+\\
    +q^2\max\set{1,A}(1+\normC{n}{1})^2\norm{\frac{\partial G_j}{\partial a}}^2 \biggr) = \\
=q\max\set{1,A}\norm{\Big(\frac{\partial F}{\partial t}\Big)\inv}\biggl(2A^2+4Aq(1+\normC{n}{1})\norm{\frac{\partial G_j}{\partial a}}+q^2(1+\normC{n}{1})^2\norm{\frac{\partial G_j}{\partial a}}^2\biggr).
    \end{gather*}
    Using the estimates \eqref{eq:NormOfInverse} and \eqref{eq:BoundsFirstOrderDerG} we obtain the bound
    \begin{gather*}
    M\leq C q^{13} \max\set{A,A^{4}} (1+\normC{n}{1})^4\implies \\
    |X_j|\leq  C q^{15}\max\set{A,A^{4}}\normC{n}{1}(1+\normC{n}{1})^4.
    \end{gather*}
    Therefore we see that \eqref{eq:DesiredErrorTermXZ} is bounded by $C\max\set{1,A^{6}}q^{30}\normC{n}{1}^2(1+\normC{n}{1})^8$, as required.

To finish the proof, we also want to claim that the implicit function $G$ that we used in our proof exists in a not too small neighborhood (it is defined for all $\normC{n}{1}\leq N\min\set{1,A^{-4}}q^{-8.5}$, hence for $|a_j|,|b_j| \leq N\min\set{1,A^{-4}}q^{-8.5}$).
Moreover, we still need to explain why the estimate \eqref{eq:NormOfInverse}  holds in this neighborhood. 
For that we use the result of \cite[Theorem 3.3]{Jindal_2023}  which gives an explicit estimate for the size of the neighborhood where the implicit function is defined.
Note that to use this result it is enough to work with upper bounds on all the norms of the derivatives, and not necessarily with the actual suprema.
We already found an upper bound for all the derivatives above.
We therefore have, using the notation of \cite{Jindal_2023},

\[\begin{cases*}
    L_x=2\max\set{A^2,A^3}\sqrt{q}, \\
    M_y = 100A\inv q^3, \\
    K_{xx}=K_{xy}=K_{yy}= 2q\max\set{1,A^3}(1+\normC{n}{1})^2.
\end{cases*}\]
Therefore, radii $r_x,r_y$  for which the solutions exist satisfy the following two inequalities:
\[
\begin{cases}
    \frac{1}{2}K_{xx}(r_x+r_y)^2 \leq \frac{r_y}{M_y}-r_x L_x, \\
    K_{xx}(r_x+r_y) \leq \frac{1}{M_y}.
\end{cases}\]
It is enough to find $r_x$  and $r_y$  for which there is equality in both inequalities. 
In this case we can use the second equation with the first one to reduce the problem to two linear equations:
\[\begin{cases}
    -L_x r_x +\frac{1}{M_y}r_y = \frac{1}{2M_y^2 K_{xx}},\\
    r_x+r_y = \frac{1}{M_y K_{xx}}.
\end{cases}\]
We can solve it,
\begin{equation}\label{eq:exprRx}
r_x = \frac{\frac{1}{2M_y^2 K_{xx}}}{\frac{1}{M_y}+L_x}=\frac{1}{2K_{xx}M_y(L_xM_y+1)},    
\end{equation}
\begin{equation}\label{eq:exprry}
r_y = \frac{\frac{2L_xM_y+1}{2M_y^2K_{xx}}}{\frac{1}{M_y}+L_x}=\frac{2L_xM_y+1}{2K_{xx}M_y(L_xM_y+1)}.
\end{equation}

The deonminator is of the magnitude of $q^{7.5}(1+\normC{n}{1})^3\max\set{A^4,1}$.
Thus, we can find a universal constant $N$  such that if $\normC{n}{1}\leq N\min\set{1,A^{-4}}q^{-8}$, then in the computations above, all coordinates of the vectors $a$  and $b$  we use are bounded by the same constant, and hence the norm of $(a,b)$  is smaller than $r_x$, and the implicit functions we use will, in fact, exist.
Moreover, the numerator for $r_y$ is of the order $\max\set{A,A^2}q^{3.5}(1+\normC{n}{1})$, so $r_y$ is of the magnitude
\[r_y\leq N (1+\normC{n}{1})^{-2}q^{-4}\min\set{A\inv,A^{-2}}.\]
To show \eqref{eq:NormOfInverse}  holds for $\norm{a}_\infty,\norm{b}_\infty\leq r_x$, $\norm{t-\mathbf{t_0}}_\infty \leq r_y$, we write the following:
\begin{equation}\label{eq:DeviationDfDt}
\frac{\partial F}{\partial t}(a,b,t)=\frac{\partial F}{\partial t}(0,0,\mathbf{t_0})+\frac{\partial^2 F}{\partial a \partial t}a+\frac{\partial^2 F}{\partial b\partial t}b+\frac{\partial ^2F}{\partial t^2}(t-\mathbf{t_0}).    
\end{equation}

Here the second order derivatives are evaluated at some arbitrary points. 
Denote by $V$  the sum of all second order derivatives, and by $W:=\frac{\partial F}{\partial t}(0,0,\mathbf{t_0})$.  
We claim that $\norm{V}=o(q^{-3})$, and that this is enough to get \eqref{eq:NormOfInverse}. 
Indeed, we know that for $a,b,t$  for which $\norm{a}_\infty,\norm{b}_\infty\leq r_x$ and $\norm{t-\mathbf{t_0}}_\infty<r_y$  the matrix $\frac{\partial F}{\partial t}(a,b,t)$ is invertible, and we can compute its inverse using the following standard method,
\begin{gather*}
    (\frac{\partial F}{\partial t}(a,b,t))\inv=(W+V)\inv = (I+W\inv V)\inv W\inv = \\
    =\sum_{k=0}^\infty (-W\inv V)^k W\inv.
\end{gather*}
We used here the fact that $\norm{V}=o(q^{-3})$ and that $\norm{W\inv}\leq 100A\inv q^3$, so for $q$  large enough $W\inv V$ has norm less than $1$. 
As a result,
\[\norm{(W+V)\inv}=\norm{\sum_{k=0}^\infty (-W\inv V)^k W\inv }\leq \norm{W\inv}\sum_{k=0}^\infty \norm{W\inv V}^k=\frac{\norm{W\inv}}{1-\norm{W\inv V}}.\]
Since $\norm{W\inv} = O(q^3)$  and if indeed $\norm{V}=o(q^{-3})$ then the denominator is bounded by a universal constant, so we indeed recover \eqref{eq:NormOfInverse} in the neighborhood $\norm{a},\norm{b}<r_x$, $\norm{t-\mathbf{t_0}}<r_y$.

Thus, to finish the proof it is enough to show required estimate on $\norm{V}$.
Indeed, our earlier computations of the second order derivatives of $F$  show that the $L^\infty$  norms of three second order derivatives in \eqref{eq:DeviationDfDt}  are bounded by 
\[2\max\set{1,A^2}(\norm{b}+(1+\norm{a})(3+\norm{b}+\norm{a})).\]
But all of these matrices are sparse, so the operator norm is bounded by 
\[2\sqrt{q}\max\set{1,A^2}(\norm{b}+(1+\norm{a})(3+\norm{b}+\norm{a})).\]
For $q$  larger than some absolute number, $r_x<1$ so the operator norms are bounded by 
\[22\max\set{1,A^2}\sqrt{q},\]
and finally we can estimate, using \eqref{eq:exprRx} , and \eqref{eq:exprry}:
\[\norm{V}\leq 22\max\set{1,A^2}\sqrt{q}(r_x+r_y)\leq 44N(1+\normC{n}{1})^{-2}\min\set{1,A^{-2}}q^{-3.5},\]
which is of the desired magnitude.
\end{proof}
\subsection{Fourier estimates for deformation function}\label{subsec:FourierTerms}
Our next goal would be to analyze the Fourier coefficients of a function $n$, which is $C^1$-small, and for which the domain $\eqref{eq:RadialDeformationDomain}$ has integrable symplectic billiards.
Just like in \cite{avila2016integrable}, the idea is that Fourier coefficients of high harmonics will be small because of the smoothness of $n$, while coefficients of smaller harmonics will be small because of the existence of the invariant curve of $q$ -periodic orbits.
The former fact is shown in equation \eqref{eq:FourierInverseQDecay}, and the latter is proved in the next lemma.
\begin{lemma}\label{lem:FourierDecayIntermediate}
 Let $\Omega$   be a domain as in $\eqref{eq:RadialDeformationDomain}$, where $n$ is  $C^{1}$ smooth, $\normC{n}{1}<1$, and $A=\alpha(\mathcal{E})$. 
  Then there exists $C = C(\normC{n}{1})$ , monotone in $\normC{n}{1}$, and a constant $N$ such that for any $q>2$, if the symplectic billiards in $\Omega$  admit an invariant curve of $q$ periodic orbits, and $\normC{n}{1}\leq N\min\set{1,A^{-4}}q^{-8}$ ,  then

 \begin{equation}\label{eq:FourierIntermediate}
 \Big|\int_0^{2\pi A} n(t)e^{\pm \frac{iqt}{A}}dt\Big| \leq C \max\set{A,A^{7}}q^{31}\normC{n}{1}^2.    
 \end{equation}
 \end{lemma}
\begin{proof}
    Since we have an invariant curve of $q$-periodic orbits, the action of these orbits is constant.
    This means that for any $t_0$, if we denote by $t_1(t_0),...,t_{q-1}(t_0)$ the points of the $q$-periodic orbit determined by $t_0$, then the function
    \[A_q(t_0)=\sum_{j=0}^{q-1} \omega(\gamma(t_j(t_0)),\gamma(t_{j+1}(t_0)))\]
    is constant.
    According to Lemma \ref{lem:quadraticEstimate}, the assumption on $\normC{n}{1}$ and $q$ implies that
    \[\Big|A_q(t_0)-A^3 q\sin\frac{2\pi}{q}-2A^3\sin\frac{2\pi}{q}\sum_{j=0}^{q-1}n(t_0+\frac{2\pi Aj}{q})\Big|\leq C\max\set{A^3,A^{9}}q^{31}\normC{n}{1}^2(1+\normC{n}{1})^8.\]
    Hence we can write, when remembering that non-zero Fourier coefficients of constants are zero,
    \begin{gather*}
        \Big|\int_0^{2\pi A} 2A^3\sin\frac{2\pi}{q}\sum_{j=0}^{q-1}n(t+\frac{2\pi Aj}{q})e^{\frac{iqt}{A}}dt\Big| = \Big|\int_0^{2\pi A}\Big(2A^3 \sin\frac{2\pi}{q}\sum_{j=0}^{q-1}n(t+\frac{2\pi Aj}{q}) + \\
        + A^3q\sin\frac{2\pi}{q}-A_q (t)\Big)e^{\frac{iqt}{A}}dt\Big|\leq \int_0^{2\pi A}\Big|2A^3 \sin\frac{2\pi}{q}\sum_{j=0}^{q-1}n(t+\frac{2\pi Aj}{q}) + \\
        +A^3q\sin\frac{2\pi}{q}-A_q(t)\Big|dt\leq C\max\set{A^4,A^{10}}q^{31}\normC{n}{1}^2(1+\normC{n}{1})^8.
    \end{gather*}
    Next, we observe that 
    \begin{gather*}
        \int_0^{2\pi A} \sum_{j=0}^{q-1}n(t+\frac{2\pi Aj}{q})e^{\frac{iqt}{A}}dt = \sum_{j=0}^{q-1}\int_0^{2\pi A} n(t+\frac{2\pi Aj}{q})e^{\frac{iqt}{A}}dt = \\
        =\sum_{j=0}^{q-1}\int_0^{2\pi A}n(t)e^{\frac{iq(t-\frac{2\pi Aj}{q})}{A}}dt=\sum_{j=0}^{q-1}\int_0^{2\pi A}n(t)e^{\frac{iqt}{A}}dt=\\
        =q\int_0^{2\pi A}n(t)e^{\frac{iqt}{A}}dt.
    \end{gather*}
    So, in total, we get:
    \[2qA^3\sin\frac{2\pi}{q}\Big|\int_0^{2\pi A}n(t)e^{\frac{iqt}{A}}dt\Big|\leq C\max\set{A^4,A^{10}}q^{31}\normC{n}{1}^2(1+\normC{n}{1})^8.\]
    Since $q\sin\frac{2\pi}{q}$ is bounded away from $0$ for $q\in\N$, and the dependence on $\normC{n}{1}$  is clearly monotone, we get the desired inequality. 
    The proof for $-q$  is identical.
\end{proof}
\subsection{Approximation argument}\label{subsec:EndOfDynThm}

The next goal would be to estimate various norms of the ``non-elliptic" part of the function $n$.
Because the function $n$  is smooth enough, there exists a universal constant $C>0$  such that for all $q\neq 0$:
\begin{equation}\label{eq:FourierInverseQDecay}
    \Big|\int_0^{2\pi A} n(t)e^{\frac{iqt}{A}}dt\Big|\leq \frac{C A^2\normC{n}{1}}{|q|}.
\end{equation}
Note that we can use Lemmas 16-19 of \cite{avila2016integrable}  as they are, by the following reasoning:
using an appropriate linear map, we can assume that the initial ellipse is the circle of perimeter $1$. 
Linear change of coordinates preserves the affine normal vector, and hence the affine normal after change of coordinates will be the affine normal to the circle.
In the case of the circle, the affine normal coincides with the usual normal. 
Hence, this linear map transforms $\Omega$  into a radial deformation of the circle, with the same radial deformation function $n$.
Now we can apply the results of Lemmas 16-19 to get estimates on $n$. 
Note that the functions in \cite{avila2016integrable}  are given in Lazutkin coordinates. 
Nevertheless, we use the lemmas for a circle, in which case the Lazutkin parametrization and the usual trigonometric parametrization coincide. 
Finally, we do need to make sure to rescale the parameter to be in $[0,1]$, as it is in \cite{avila2016integrable}, and not in $[0,2\pi A]$, which we used so far.
This only has effect on the norms of the derivatives of the functions we consider. 
The result is the following.
\begin{lemma}\label{lem:EllipticTermC1Norm}
    Let $\mathcal{E}$  be an ellipse, and $n\in\mathrm{span}\set{e^{\frac{iqt}{A}}\mid |q|\leq 2}$ , with sufficiently small coefficients.  
    Then there exists $C>0$, and an ellipse $\bar{\mathcal{E}}$, such that if $\bar{\mathcal{E}}=\mathcal{E}+n_{\bar{\mathcal{E}}}$ , then 
    \[\normC{n-n_{\bar{\mathcal{E}}}}{1}\leq C\max\set{A^2,A^{-1}}\normC{n}{1}^2,\]
    where as before, $A=\alpha(\mathcal{E})$.
\end{lemma}
As can be seen from the proof, there is also an affine map $T$  for which $T(\mathcal{E})=\bar{\mathcal{E}}$. and moreover $\norm{T-I}$  is comparable to $\normC{n}{1}$.

Next, we derive an analogue of Lemma 7 of \cite{avila2016integrable}.
The idea is to compare the deformation functions $n$  that correspond to the same domain with respect to two nearby ellipses.
The proof follows the same ideas of Lemma 7 in \cite{avila2016integrable}.
\begin{lemma}\label{lem:Lemma7ADK}
Suppose that a domain $\Omega$ is given by a deformation of an ellipse $\mathcal{E}$ as in \eqref{eq:RadialDeformationDomain}, and that $\bar{\mathcal{E}}$  is an ellipse close to $\mathcal{E}$.
Namely, if $T$  is an affine map for which $T(\bar{\mathcal{E}})=\mathcal{E}$, then $\norm{T-I}\leq \frac{1}{100}$.
Let $n_{\bar{\mathcal{E}}}$  be the function for which 
\[\bar{\mathcal{E}}=\mathcal{E}+n_{\bar{\mathcal{E}}},\]
and $\bar{n}$   be the function for which
\[\Omega=\bar{\mathcal{E}}+\bar{n}.\]
Then if $\normC{n_{\bar{\mathcal{E}}}}{1}$ is small enough, then there exists a universal constant $C$ for which
\begin{equation}
    \label{eq:SameDomainDifferentEllipses}
    \normC{\bar{n}}{1}\leq C \normC{n-n_{\bar{\mathcal{E}}}}{1}.
\end{equation}

\end{lemma}
\begin{proof}
Let $T$  be an affine map for which $T(\bar{\mathcal{E}})=\mathcal{E}$.
We can write $T(x)=v_0 +Ax$ for some vector $v_0$  and a linear map $A$. The assumption that $\norm{T-I}\leq\frac{1}{100}$  implies that $A$ is close to the identity, and $v_0$ is close to zero.
Now, fix an inner product on $\R^2$ according to which $\mathcal{E}$  is the unit circle.
Then, the ellipse $\bar{\mathcal{E}}$  can also be described as $\set{v_0+r(t)e_t\mid t\in[0,2\pi]}$ for some function $r$  which is close in $C^1$ norm to the constant function $1$  (and in fact, $\normC{r-1}{1}$ is comparable with $\norm{T-I}$), and $e_t=(\cos t,\sin t)$. 
Given a point $X$  close enough to $\mathcal{E}$  and to $\bar{\mathcal{E}}$, there are unique $t,n,\bar{t},\nu$  for which 
\begin{equation}\label{eq:ChangePerspectiveOfPoint}
X=(1+n)e_t=v_0+(1+\nu)r(\bar{t})e_{\bar{t}},
\end{equation}
see Figure \ref{fig:DeformationComparison}.
\begin{figure}
        \centering
        \begin{tikzpicture}[scale = 1.75]
        	\begin{scope}[decoration={
markings,
mark = at position 1 with {\arrow{>}}}
]
\draw[black, domain = -180:180, variable = \t, line width = 0.7mm] plot({cos(\t)},{sin( \t)});
\draw[blue, domain = -180:180, variable = \t, line width = 0.7mm] plot({0.3+0.3*(cos(\t)-2*sin(\t))},{0.1+0.7*(cos(\t)+sin( \t))});
\draw[red, domain = -180:180, variable = \t, line width = 0.7mm] plot({cos(\t)*(1+0.2*cos(\t)-0.1*cos(2*\t))},{sin(\t)*(1-0.4*sin( \t)+0.01*cos(3*\t))});
\tkzDefPoint(0,0){O};
\tkzDrawPoint(O);
\tkzDefPoint(0.3,0.1){E};
\tkzDrawPoint(E);
\draw[dashed,black,line width = 0.3mm] (0,0)--(0,-1.4);
\draw[dashed,blue,line width = 0.3mm] (0.3,0.1)--(0,-1.4);
\node at (-0.15,-0.2) {$e_t$};
\node[blue] at (0.5, -0.3) {$e_{\bar{t}_\Omega(t)}$};
\draw[green, postaction = {decorate}, line width = 0.8mm](0,0)--(0.3,0.1);
\draw[black,dashed] (0,0)-- (-1.2, 0.3);
\draw[blue,dashed](0.3,0.1)--(-1.2,0.3);
\tkzDefPoint(-1.2,0.3){C};
\tkzDrawPoint(C);
\node[red] at (-1.3,0.3) {$X$};
\node[green] at (0.4,0.2) {$v_0$};
\node[below,red] at (0,-1.4) {$\Omega$};
\node [black]  at (-0.8,-0.8) {$\mathcal{E}$};
\node[blue] at (0,1.2) {$\bar{\mathcal{E}}$};
\end{scope}
       \end{tikzpicture}
       \caption{Same domain ($\Omega$, in red) seen as a deformation of two different ellipses, $\mathcal{E}$ (in black) and $\bar{\mathcal{E}}$ (in blue). Here $X=(1+n)e_t = v_0+(1+\nu)r(\bar{t})e_{\bar{t}}$ (equation \eqref{eq:ChangePerspectiveOfPoint}). The vectors pointing downwards demonstrate the definition of $\bar{t}_\Omega$.\label{fig:DeformationComparison}}
  \end{figure}
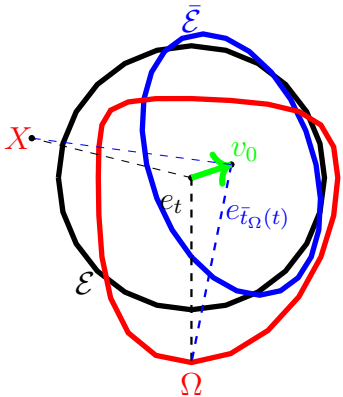
This gives rise to a diffeomorphism $\psi(t,n)=(\bar{t}(t,n),\nu(t,n))$. 
Observe that $\psi(t,n_{\bar{\mathcal{E}}}(t))=(\bar{t}(t,n_{\bar{\mathcal{E}}}(t)),0)$, or equivalently, that
\begin{equation}\label{eq:NewEllipseFromOldOne}
    (1+n_{\bar{\mathcal{E}}}(t))e_t=v_0+r(\bar{t}(t,n_{\bar{\mathcal{E}}}(t)))e_{\bar{t}(t,n_{\bar{\mathcal{E}}})}.
\end{equation}
We also get a function $\bar{t}_\Omega(t)$ which is defined by (see again Figure \ref{fig:DeformationComparison})
\[\psi(t,n(t))=(\bar{t}_\Omega(t),\bar{n}(\bar{t}_\Omega(t))),\] for all $t$, and it holds that
\begin{equation}\label{eq:NuAndN}
\nu(t,n(t))=\bar{n}(\bar{t}_\Omega(t)).    
\end{equation}
We can also isolate from \eqref{eq:ChangePerspectiveOfPoint} and \eqref{eq:NewEllipseFromOldOne},
\[\nu(t,n)=\frac{1}{r(\bar{t}(t,n))}\norm{(n-n_{\bar{\mathcal{E}}}(t))e_t + r(\bar{t}(t,n_{\bar{\mathcal{E}}}(t)))e_{\bar{t}(t,n_{\bar{\mathcal{E}}}(t)))}}-1.\]
This means that 
\[\nu(t,n)=\frac{1}{r(\bar{t}(t,n))}(n-n_{\bar{\mathcal{E}}}(t))(1+\rho(t,n)),\]
where $\rho$  is a function whose norm is comparable to $\normC{n_{\bar{\mathcal{E}}}}{1}$. 
Pick a number $t_0$  for which $v_0$ is a positive multiple of  $e_{t_0}$. 
Then it holds that $\bar{t}(t_0,n)=t_0$ for all $n$, and in particular $\bar{t}_\Omega(t_0)=t_0$. 
Moreover, equations \eqref{eq:ChangePerspectiveOfPoint}, \eqref{eq:NewEllipseFromOldOne} imply that for all $n$,
\begin{equation}\label{eq:AtSpecialDirection}
    r(t_0)\nu(t_0,n)=n-n_{\bar{\mathcal{E}}}(t_0).
\end{equation}
Now, for any $\bar{t}$  we can write, using \eqref{eq:NuAndN}
\begin{multline}\label{eq:DerivationNBar}
    \bar{n}(\bar{t})=\bar{n}(\bar{t}_\Omega(t_0))+\int_{\bar{t}_\Omega(t_0)}^{\bar{t}}\bar{n}'(\xi)d\xi=\bar{n}(\bar{t}_\Omega(t_0))+\int_{t_0}^{\bar{t}_\Omega\inv(\bar{t})}\bar{n}'(\bar{t}_\Omega(\eta))\bar{t}'_\Omega(\eta)d\eta=\\
    \underset{\eqref{eq:NuAndN}}{=}\bar{n}(\bar{t}_\Omega(t_0))+\int_{t_0}^{\bar{t}_\Omega\inv(t)}\frac{d}{d \eta}\nu(\eta,n(\eta)) d\eta = \bar{n}({\bar{t}_\Omega(t_0)})+\\
    +\int_{t_0}^{\bar{t}\inv_\Omega(t)}[\frac{d}{d\eta}\frac{1}{r(\bar{t}(\eta,n(\eta)))}(n(\eta)-n_{\bar{\mathcal{E}}}(\eta))](1+\rho)+\frac{1}{r}(n-n_{\bar{\mathcal{E}}})(\frac{\partial\rho}{\partial t}+\frac{\partial \rho}{\partial n}n'(\eta))d\eta=\\
    =\bar{n}({\bar{t}_\Omega(t_0)})+\frac{1}{r(\bar{t}(\bar{t}\inv_\Omega(\bar{t}),n(\bar{t}\inv_\Omega(\bar{t})))}(n(\bar{t}\inv_\Omega(\bar{t}))-n_{\bar{\mathcal{E}}}(\bar{t}\inv_\Omega(\bar{t})))-\\
    -\frac{1}{r(\bar{t}_\Omega(t_0))}(n(t_0)-n_{\bar{\mathcal{E}}}(t_0))+\int_{t_0}^{\bar{t}_\Omega\inv(t)}\frac{1}{r(\bar{t}(\eta,n(\eta)))}(n(\eta)-n_{\bar{\mathcal{E}}}(\eta))(\frac{\partial\rho }{\partial t}+\frac{\partial\rho}{\partial n}n'(\eta))d\eta+\\
    +\int_{t_0}^{\bar{t}_\Omega\inv(t) }\rho\cdot \frac{d}{d\eta}\frac{1}{r(\bar{t}(\eta,n(\eta)))}(n(\eta)-n_{\bar{\mathcal{E}}}(\eta))d\eta
\end{multline}
The integrals in \eqref{eq:DerivationNBar}  are of the order of $\normC{n-n_{\bar{\mathcal{E}}}}{1}\normC{n_{\bar{\mathcal{E}}}}{1}$.
The terms that involve $t_0$  cancel using  \eqref{eq:AtSpecialDirection}.
So we can simplify \eqref{eq:DerivationNBar}:
\[\bar{n}(\bar{t})=\frac{1}{r(\bar{t}(\bar{t}\inv_\Omega(\bar{t}),n(\bar{t}\inv_\Omega(\bar{t})))}(n(\bar{t}\inv_\Omega(\bar{t}))-n_{\bar{\mathcal{E}}}(\bar{t}\inv_\Omega(\bar{t})))+O(\normC{n-n_{\bar{\mathcal{E}}}}{1}\normC{n_{\bar{\mathcal{E}}}}{1}),\]
which gives us the estimate we want. 
\end{proof}
In what follows, the main idea  is the same idea as in \cite{avila2016integrable}, and was mentioned in item \ref{itm:Approximation}  at the start of the section: if we have a rationally integrable domain close enough to an ellipse, then there is another ellipse that is much closer to our domain.
This is made percise in the following lemma.
\begin{lemma}\label{lem:FindBetterApproxEllipse}
    Let $\Omega$  be a rationally integrable domain, given by the formula $\eqref{eq:RadialDeformationDomain}$, where we assume that the function $n$  is $C^{127}$ smooth, and its $C^1$ norm is sufficiently small (so, e.g., \eqref{eq:AsymptoticCatMouse} holds). 
    Then there exists an ellipse $\bar{\mathcal{E}}$, such that if
    \[\partial\Omega = \bar{\mathcal{E}}+\bar{n}\]
    then
    \begin{equation}\label{eq:nBarFunnyPower}
        \normC{\bar{n}}{1}\leq C(\normC{n}{127})\max\set{A^{-129},A^{132}}\normC{n}{1}^{7875\slash 7874}.
    \end{equation}
    Here $C(\normC{n}{127})$  denotes some monotone function of $\normC{n}{127}$, and $A=\alpha(\mathcal{E})$.
    \end{lemma}
    \begin{proof}
        First, we decompose,
        \[n=n_{ell}+n^\perp,\]where $n_{ell}\in\mathrm{span}\set{e^{\frac{iqt}{A}}\mid |q|\leq 2}$, and $n^\perp\in\mathrm{span}\set{e^{\frac{iqt}{A}}\mid |q|>2}$.
   We can also write $n_{ell}(t)=\sum_{k=-2}^2 a_k e^{\frac{ikt}{A}}$ . 
   By orthogonality,
   \[\normL{n_{ell}}{2}^2\leq \normL{n_{ell}}{2}^2+\normL{n^\perp}{2}^2=\normL{n}{2}^2\leq C\normC{n}{1}^2.\]
For any $m\in\N$,  we can find a constant $L_m$  for which   \begin{equation}\label{eq:99NormC1NormAnalyticity}
  \normC{n_{ell}}{m}\leq L_m A^{-(m+1)}\normC{n}{1}.
  \end{equation}
  Indeed,
  \[|n_{ell}^{(m)}(t)|=|\sum_{k=-2}^2(\frac{ik}{A})^m a_k e^{\frac{ikt}{A}}|\leq (\frac{2}{A})^m\sum_{k=-2}^2 |a_k|, \]
  and thus
  \begin{gather*}
  \normC{n_{ell}}{m}\leq  (\frac{2}{A})^{m+1}\sum_{k=-2}^2 |a_k|\leq \sqrt{5} (\frac{2}{A})^{m+1}\sqrt{\sum_{k=-2}^2 |a_k|^2}\leq \\
  \leq L_mA^{-(m+1)}\normL{n}{2}\leq L_m A^{-(m+1)}\normC{n}{1}.
  \end{gather*}
   Let $\bar{\mathcal{E}}$  the ellipse the corresponds to $n_{ell}$, using Lemma \ref{lem:EllipticTermC1Norm}.
   Then, there exists a function $\bar{n}$  such that $\partial\Omega = \bar{\mathcal{E}}+\bar{n}$, and there exists a function $\tilde{n}$  such that $\bar{\mathcal{E}}=\mathcal{E}+\tilde{n}$.
   Then, Lemma \ref{lem:EllipticTermC1Norm} gives a universal constant for which
   \begin{equation}\label{eq:GeneratedBy5HarmonicsREALLYClose}
       \normC{n_{ell}-\tilde{n}}{1}\leq C\max\set{A\inv,A^{2}}\normC{n_{ell} }{1}^2\leq C \max\set{A^{-3},1}\normC{n}{1}^2. 
   \end{equation}
Lemma 19 of \cite{avila2016integrable}  also implies that $\bar{\mathcal{E}}$  is a linear image of $\mathcal{E}$ which is $\normC{n_{ell}}{1}$  close to identity.
   Using \eqref{eq:99NormC1NormAnalyticity}  and \eqref{eq:GeneratedBy5HarmonicsREALLYClose}, for $n$  small enough, we get $\normC{\tilde{n}}{1}\leq\frac{1}{10}$. 
   Thus, we can use Lemma \ref{lem:Lemma7ADK}  to get
   \[\normC{\bar{n}}{1}\leq C\normC{n-\tilde{n}}{1}.\]
   Therefore, we can estimate
   \begin{multline}\label{eq:nBarNPerpNTilde}
   \normC{\bar{n}}{1}\leq C\normC{n-\tilde{n}}{1}=\normC{n_{ell}+n^\perp-\tilde{n}}{1}\leq \\
\leq    C\max\set{1,A^{-3}}(\normC{n^\perp}{1}+\normC{n_{ell}-\tilde{n}}{1}).
   \end{multline}
   Consequently, our next goal is to estimate $\normC{n^\perp}{1}$.
   Using Parseval's identity,
   \[\normL{n^\perp}{2}^2=\sum_{|q|\geq 3} |a_q|^2,\]where $a_q=\frac{1}{2\pi A}\int_0^{2\pi A} n(t)e^{-\frac{iqt}{A}}dt$  are the Fourier coefficients of $n$. 
   Take $\alpha = \frac{1}{31}$, and $q_0 = \lfloor \normC{n}{1}^{-\alpha}\rfloor$.
   Then, by \eqref{eq:FourierIntermediate}, we have, for a constant depending on $\normC{n}{1}$,
   \[|a_q|\leq C(\normC{n}{1})\max\set{1,A^{6}}|q|^{31}\normC{n}{1}^2\leq C(\normC{n}{1})\max\set{1,A^{6}}\normC{n}{1}^{2-31\alpha},\]for all $3\leq|q|\leq q_0$.
   Here, $C(\normC{n}{1})$ is a monotone function of $\normC{n}{1}$.
   To use \eqref{eq:FourierIntermediate} for all of these values of $q$, we must make sure that the assumption $\normC{n}{1}\leq N\min\set{1,A^{-4}}|q|^{-8}$ is satisfied.
   The strictest requirement is for $q=q_0$ , in which case we get the inequality
   \begin{equation}\label{eq:AsymptoticCatMouse}
       \normC{n}{1}\leq N\min\set{1,A^{-4}} \normC{n}{1}^{8\alpha}\Longleftrightarrow\normC{n}{1}^{1-8\alpha}\leq N\min\set{1,A^{-4}}.
   \end{equation}
   For our choice of $\alpha$,  the power of $\normC{n}{1}$ is positive, hence, the inequality holds for $\normC{n}{1}$ small enough.
   Now, we can sum these inequalities and get
   \begin{equation}\label{eq:SumFourierIntermediate}
       \sum_{3\leq |q| \leq q_0} |a_q|^2 \leq C(\normC{n}{1})\max\set{1,A^{12}}\normC{n}{1}^{4-63\alpha}.
   \end{equation}
   In addition, for $|q|>q_0$, using \eqref{eq:FourierInverseQDecay}, we have
   \begin{equation}\label{eq:SumFourierLarge}
       \sum_{|q|>q_0} |a_q|^2 \leq CA^2 \normC{n}{1}^{2+\alpha}.
   \end{equation}
   Our choice of $\alpha$  guarantees that $4-63\alpha = 2+\alpha = 63\slash 31$, so we can use \eqref{eq:SumFourierIntermediate}  and \eqref{eq:SumFourierLarge} to get
   \begin{multline}\label{eq:NormOfNonElliptic}
       \normL{n^\perp}{2}^2\leq C(\normC{n}{1})\max\set{1,A^{12}}\normC{n}{1}^{63\slash 31}\implies\\
       \normL{n^\perp}{2}\leq C(\normC{n}{1})\max\set{1,A^{6}}\normC{n}{1}^{63\slash 62}.
   \end{multline}
   Now, we use the Sobolev interpolation inequalities (see, e.g., \cite[Theorem 7.28]{gilbarg1977elliptic}).
   We slightly modify the inequality by rescaling to account for our functions having domain, the length of which depends on $A$.
   There exists an absolute constant $C>0$, such that for all $\varepsilon>0$ and $C^{127}$ functions (like $n^\perp$), for $j=1,2$ we have
   \begin{equation}\label{eq:SobolevIneq}
       \normL{(n^\perp)^{(j)}}{2}\leq C\max\set{A^{-j},A^{127-j}}(\varepsilon\normC{n^\perp}{127}+\varepsilon^{\frac{j}{j-127}}\normL{n^\perp}{2}).
   \end{equation}
   In particular, for $\varepsilon=\normC{n^\perp}{1}^{7875\slash 7874}$ , using \eqref{eq:NormOfNonElliptic}  in \eqref{eq:SobolevIneq}  for $j=2$ (which gives us the stricter inequality) we get
   \[\normL{(n^\perp)^{\prime\prime}}{2}\leq  C(\normC{n}{1})\max\set{A^{131},A^{-2}}\normC{n}{1}^{7875\slash 7874}(1+\normC{n^\perp}{127}).\]And the same right-hand side also bounds $\normL{(n^\perp)^\prime}{2}$.  
   With this we can conclude that
   \begin{gather*}
       \normC{n^\perp}{1}\leq 2\pi A(\normL{(n^\perp)^\prime}{2}+\normL{(n^\perp)^{\prime\prime}}{2})\leq\\
    \leq    C(\normC{n}{1})\max\set{A^{-1},A^{132}}(1+\normC{n^\perp}{127})\normC{n}{1}^{7875\slash 7874}.
   \end{gather*}
Using \eqref{eq:99NormC1NormAnalyticity}, we can bound the above with a constant that depends only on $\normC{n}{127}$. 
Indeed,
\[\normC{n^\perp}{127}=\normC{n-n_{ell}}{127}\leq \normC{n}{127}+\normC{n_{ell}}{127}\leq \normC{n}{127}+CA^{-128}\normC{n}{1},\]
and $\normC{n}{1}$ is small, so this can be bounded by $\max\set{1,A^{-128}}(B+\normC{n}{127})$, for some universal constant $B$. 
Also, since $\normC{n}{1}\leq\normC{n}{127}$, then the constant of \eqref{eq:SumFourierIntermediate} which was monotone in $\normC{n}{1}$  can be replaced with a monotone function of $\normC{n}{127}$.    
   Now, we can use this inequality in \eqref{eq:nBarNPerpNTilde}  to see that
   \[\normC{\bar{n}}{1}\leq C(\normC{n}{127})\max\set{A^{-129},A^{132}}\normC{n}{1}^{7875\slash 7874}.\]
This is exactly \eqref{eq:nBarFunnyPower}.   \end{proof}
\subsection{Finishing the proof of Theorem \ref{thm:SympDynRigidity}}\label{subsec:FinishProofOfLocalBirkhoff}
Let $\mathcal{E}$ be any ellipse in $\R^2$, and write $A=\alpha(\mathcal{E})$. 
We first consider the family of ellipses
\[P = \set{\mathcal{E}
\mid d_\Delta(\mathcal{E},\mathcal{E}')\leq \frac{1}{10}A^3},\]
where $d_\Delta$  denotes the symmetric difference metric on compact sets in $\R^2$:
\[\forall X,Y\subseteq \R^2, \, d_\Delta (X,Y)=\mathrm{area}(X\Delta Y).\]
Since all ellipses can be described using five real parameters, we can see this subset as a compact subset of $\R^5$.
Then, we can find $\varepsilon>0$ with the following property: for any function $n$  which is $C^1$ $\varepsilon$-small and $C^{127}$ $K$-small (where $K$  is the constant in the formulation of the theorem), for any ellipse $\mathcal{E}'\in P$, if $\bar{n}$  is the function for which
\[\mathcal{E}+n=\mathcal{E}'+\bar{n},\]
then $\normC{\bar{n}}{127}\leq 2K$.
By shrinking $\varepsilon$ further, we may also assume that 
\begin{enumerate}
    \item $C(2K)\max\set{A^{-129},A^{132}}\varepsilon^{{7875}\slash{7874}}<\frac{1}{200}\varepsilon$, \label{itm:HigherPower}
    \item $\varepsilon<\frac{1}{100}$, \label{itm:JustSmall}
    \item $\varepsilon^{1-\frac{8}{31}}\leq N\min\set{1,A^{-4}}$. \label{itm:CatMouseSolution}
\end{enumerate}
Here, $C$ is the monotone function of \eqref{eq:nBarFunnyPower}, and $N$  is the universal constant that appears in Lemma \ref{lem:quadraticEstimate} (see also \eqref{eq:AsymptoticCatMouse}). 
Now consider the collection of ellipses
\[E=\set{\mathcal{E'}\mid d_\Delta(\mathcal{E},\mathcal{E'})\leq 10A^3 \varepsilon},\]
We first establish a relation between the symmetric difference distance and the norm of the deformation $n$.
\begin{lemma}\label{lem:SymmetricDifferenceC1Norm}
    Suppose a domain $\Omega$  is a deformation of an ellipse $\mathcal{E}$, given by \eqref{eq:RadialDeformationDomain}.
    Then
    $d_\Delta(\mathcal{E},\Omega)\leq \alpha(\mathcal{E})^3 \normC{n}{1}(1+\normC{n}{1}).$
\end{lemma}
\begin{proof}
    Write $A=\alpha(\mathcal{E})$, and assume that $\mathcal{E}$  is parametrized by $e_{a,b}(t)=(a\cos \frac{t}{A},b\sin \frac{t}{A})$.
    Then the symmetric difference $\mathcal{E}\Delta\Omega$ can be parametrized with
    \[\mathcal{E}\Delta\Omega = \set{e_{a,b}(t)(1+\lambda n(t))\mid t\in[0,2\pi A],\lambda\in[0,1]}.\]
    In terms of $t,\lambda$, the area form is given by $A^2n(t)(1+\lambda n(t))dt\wedge d\lambda$. 
    As a result,
    \[\mathrm{area}(\mathcal{E}\Delta \Omega)=\int_{[0,2\pi A]\times [0,1]} A^2 |n(t)(1+\lambda n(t))|d\lambda dt.\]
    Thus, we immediately get the bound
    \[\mathrm{area}(\mathcal{E}\Delta\Omega)\leq 2\pi A \cdot A^2 \normC{n}{1}(1+\normC{n}{1})=2\pi A^3 \normC{n}{1}(1+\normC{n}{1}).\]
\end{proof}
Consider a domain $\Omega$  of the form \eqref{eq:RadialDeformationDomain} with $n$  as per the assumptions of the theorem.
The set of ellipses can be parametrized by five real parameters, and the condition $d_\Delta(\mathcal{E},\mathcal{E}')\leq 10A^3\varepsilon$  determines a compact set in the set of parameters.  
For each $a$  in the parameter set, denote by $\mathcal{E}_a$  the corresponding ellipse, and write $n_a$  for the radial deformation function for which 
\[\Omega = \mathcal{E}_a+n_a.\]
By compactness, we can choose a parameter $a_*$  such that $\normC{n_{a_*}}{1}$  is minimal, and we simplify notation by writing $n_*=n_{a_*}$. 
From the minimality, it follows that $\normC{n_*}{1}\leq\normC{n}{1}<\varepsilon$, and the choice of $\varepsilon$  guarantees $\normC{n_*}{127}\leq 2K$.  
For small enough $\varepsilon$, for ellipses $\mathcal{E'}\in E$ we have $\alpha(\mathcal{E'})\in [\frac{3}{4}A,\frac{5}{4}A]$. 
Now we use Lemma \ref{lem:FindBetterApproxEllipse} (which we can use thanks to the requirement \ref{itm:CatMouseSolution} on $\varepsilon$), from which we get an ellipse $\bar{\mathcal{E}}_*$  and a function $\bar{n}_*$  such that
\[\Omega = \bar{\mathcal{E}}_*+\bar{n}_*,\] and 
\[\normC{\bar{n}_*}{1}\leq C(\normC{n_*}{127})\max\set{A^{-129},A^{132}}\normC{n_*}{1}^{{7875}\slash{7874}}.\]
We used here the fact that the ratio
$\alpha(\mathcal{E}_*)\slash A$ is bounded by constants from above and from below. 
Our choices of constants guarantee that (see requirement \ref{itm:HigherPower} on $\varepsilon$)

\begin{equation}\label{eq:FoundBetterApproximate} 
\normC{\bar{n}_*}{1}\leq\frac{1}{200}\normC{n_*}{1}. 
\end{equation}
Now Lemma \ref{lem:SymmetricDifferenceC1Norm}  gives us $d_\Delta (\bar{\mathcal{E}}_*,\Omega)\leq 2\pi \alpha(\bar{\mathcal{E}}_*)^3 \normC{\bar{n}_*}{1}(1+\normC{\bar{n}_*}{1}).$
Note that $\pi\alpha(\bar{\mathcal{E}}_*)^3=\mathrm{area}(\bar{\mathcal{E}}_*)$.
Let $\tilde{n}$ denote the function for which $\bar{\mathcal{E}}_*=\mathcal{E}+\tilde{n}$.
A similar computation to that of Lemma \ref{lem:SymmetricDifferenceC1Norm}  gives
\[\mathrm{area}(\bar{\mathcal{E}}_*) \leq \pi A^3 (1+\normC{\tilde{n}}{1})^2.\]
It follows from the proof of Lemma \ref{lem:FindBetterApproxEllipse}  and Lemma \ref{lem:EllipticTermC1Norm}, 
\[\normC{\tilde{n}-n_{ell}}{1}\leq\max\set{A^2,A^{-1}}\normC{n_*}{1}^2,\]
where $n_{ell}$  is the projection of $n_*$  into the subspace $\mathrm{span}\set{e^{\frac{iqt}{A}}\mid |q|\leq 2}$. 
Moreover, from \eqref{eq:99NormC1NormAnalyticity} $\normC{n_{ell}}{1}\leq CA^{-2}\normC{n_*}{1}$, so overall 
\[\normC{\tilde{n}}{1}\leq C\max\set{A^2,A^{-2}}\normC{n_*}{1}\leq C\max\set{A^2,A^{-2}}\normC{n}{1}.\]
Overall, we get
\[d_\Delta(\bar{\mathcal{E}}_*,\Omega)\leq 2\pi A^3 (1+\normC{\tilde{n}}{1})^2\frac{1}{200}\normC{n_*}{1}(1+\normC{\bar{n}_*}{1}).\]
Using the estimates above, we get
\[d_\Delta(\bar{\mathcal{E}}_*,\Omega) \leq \frac{1}{2}A^3 \varepsilon.\]
In addition, again using Lemma \ref{lem:SymmetricDifferenceC1Norm},
\[d_\Delta(\mathcal{E},\Omega) \leq A^3 \normC{n}{1}(1+\normC{n}{1})\leq 2A^3 \varepsilon.\]
Then, the triangle inequality gives
\[d_\Delta(\mathcal{E},\bar{\mathcal{E}}_*)\leq 10A^3\varepsilon,\]
which means that $\bar{\mathcal{E}}_*\in E$. 
Then, the minimality of $n_*$ in $E$  means that $\normC{n_*}{1}\leq \normC{\bar{n}_*}{1}$, but together with \eqref{eq:FoundBetterApproximate}, we must have $n_*=0$, which means that $\Omega$ is an ellipse. 
\bibliography{NewBibliography.bib}
\bibliographystyle{abbrv}
\end{document}